%% file: main.tex
\documentclass[11pt]{article}

\RequirePackage[amsthm,amsmath]{imsart}

\usepackage[mathscr]{eucal}
\usepackage{graphicx}
\usepackage{latexsym}
\usepackage{amssymb}
\usepackage{float}
\usepackage{psfrag}
\usepackage{epsfig}
\usepackage{enumerate}
\DeclareMathAlphabet{\mathpzc}{OT1}{pzc}{m}{it}

\include{psbox}

\usepackage{amsfonts}
\usepackage{graphicx}
\usepackage{amsfonts}
\usepackage{color}
\usepackage[usenames,dvipsnames]{xcolor}
\usepackage[textsize=small]{todonotes}

\usepackage[numbers]{natbib}

\input{ex_shared}

\begin{document}

\input{definitions}

\newcommand{\beginsec}{
\setcounter{lemma}{0} \setcounter{theorem}{0}
\setcounter{corollary}{0} \setcounter{definition}{0}
\setcounter{example}{0} \setcounter{proposition}{0}
\setcounter{condition}{0} \setcounter{assumption}{0}
\setcounter{remark}{0} }

\numberwithin{equation}{section} \numberwithin{lemma}{section}

\begin{frontmatter}

\title{Mean-field type modeling of nonlocal crowd 
aversion in 
pedestrian crowd dynamics\thanks{This work is partially supported by the Swedish 
Research Council via 
Grant:~2016-04086.} }

\author{Alexander Aurell \thanks{Department of Mathematics, KTH Royal Institute 
of Technology, 100 44 
Stockholm, Sweden. \newline E-mail address: aaurell@kth.se}
\qquad 
Boualem Djehiche \thanks{Department of Mathematics, KTH Royal Institute 
of Technology, 100 44 
Stockholm, Sweden. \newline E-mail address: boualem@kth.se}
}
\runtitle{Tagged pedestrian}
\runauthor{Aurell, Djehiche}

%\today

\skp

\begin{abstract}
We extend the class of pedestrian crowd models introduced by 
Lachapelle and Wolfram (2011) to allow for nonlocal crowd 
aversion and 
arbitrarily but finitely many interacting crowds. The new crowd 
aversion feature grants pedestrians a 'personal space' where crowding is 
undesirable. We derive the model from a particle picture and 
treat it as a mean-field type game. Solutions to the mean-field type game are 
characterized via a Pontryagin-type Maximum Principle. The behavior of 
pedestrians acting under nonlocal crowd aversion is illustrated by a numerical 
simulation.

\noi {\bf MSC 49N90, 60G09, 60H10, 60H30, 60K35}

\noi {\bf Keywords: Crowd dynamics, crowd aversion, mean-field approximation, 
interacting populations, optimal control, mean-field type game}
\end{abstract}

\end{frontmatter}

\section{Introduction}
\label{sec:introduction}
\input{introduction}

\section{Preliminaries}
\label{sec:preliminaries}
\input{preliminaries}

\section{Single-crowd model for crowd aversion}
\label{sec:singlecrowd}
\input{singlecrowd}

\section{Multi-crowd model for crowd aversion}
\label{sec:multiplecrowd}
\input{multiplecrowd}

\section{Numerical example}
\label{sec:numexp}
\input{example}

\section{Appendix}
\label{sec:appendix}
\input{appendix}

\bibliographystyle{agsm}
\bibliography{references}

\end{document}

%% file: ex_shared.tex
\usepackage{lipsum}
\usepackage{amsfonts}
\usepackage{graphicx}
\usepackage{epstopdf}
\usepackage{algorithmic}
\ifpdf
  \DeclareGraphicsExtensions{.eps,.pdf,.png,.jpg}
\else
  \DeclareGraphicsExtensions{.eps}
\fi

\theoremstyle{remark}
\newtheorem{remark}{\bf Remark}[section]

\newcommand{\expv}[1]{\mathbb{E}\left[ #1 \right]}
\newcommand{\expvN}[1]{\mathbb{E}_N\left[ #1 \right]}
\newcommand{\rd}{\mathbb{R}^d}
\newcommand{\tr}[1]{\textup{Tr}\left[#1\right]}
\newcommand{\rdd}{\mathbb{R}^{d\times d}}
\newcommand{\F}{\mathcal{F}}
\newcommand{\Ad}{\mathcal{A}_d}
\newcommand{\N}{\mathbb{N}}
\newcommand{\I}{\mathbb{I}}
\newcommand{\bN}{[\![N ]\!]}
\newcommand{\bM}{[\![M ]\!]}
\newcommand{\A}{\mathcal{A}}
\newcommand{\p}{\mathbb{P}}
\newcommand{\pp}{\mathcal{P}}
\newcommand{\xN}[1]{\mathbb{E}_N\left[#1\right]}
\newcommand{\x}[1]{\mathbb{E}\left[#1\right]}
\newcommand{\ve}{\varepsilon}
\newcommand{\f}{\mathbb{F}}
\newcommand{\T}{\mathbb{T}}
\usepackage{amsopn}

%% file: definitions.tex
\newtheorem{theorem}{\bf Theorem}[section]

\newtheorem{definition}{\bf Definition}[section]
\newtheorem{corollary}{\bf Corollary}[section]

\newtheorem{lemma}{\bf Lemma}[section]
\newtheorem*{claim}{\bf Claim}
\newtheorem{assumption}{Assumption}
\newtheorem{condition}{\bf Condition}[section]
\newtheorem{proposition}{\bf Proposition}[section]
\newtheorem{definitions}{\bf Definition}[section]
\newtheorem{problem}{\bf Problem}
\numberwithin{equation}{section}
\newcommand{\skp}{\vspace{\baselineskip}}
\newcommand{\noi}{\noindent}
\newcommand{\osc}{\mbox{osc}}
\newcommand{\lfl}{\lfloor}
\newcommand{\rfl}{\rfloor}

\theoremstyle{remark}
\newtheorem{example}{\bf Example}[section]

\newcommand{\inpr}[2]{\left< #1, #2 \right>}
\newcommand{\prodrd}[2]{\left(\;#1\;,\; #2\;\right)}
\newcommand{\wx}[1]{\widetilde{\mathbb{E}}\left[ #1 \right]}
\newcommand{\indi}[1]{\mathbb{I}_{#1}}
\newcommand{\note}[1]{\begin{addmargin}[2em]{2em}\underline{Note:}\\ #1 
\end{addmargin}}
\newcommand{\conv}[3]{\left(#1 \ast #2\right)(#3)}
\newcommand{\asdd}{\,\textup{d}}
\newcommand{\asde}{\textup{d}}
\newcommand{\wtm}{\widetilde{m}}
\newcommand{\wtb}{\widetilde{b}}
\newcommand{\wtu}{\widetilde{u}}
\newcommand{\wha}{\widehat{a}}
\newcommand{\whb}{\widehat{b}}
\newcommand{\Adet}{\mathcal{A}_{\text{det}}}
\newcommand{\leqnomode}{\tagsleft@true}
\newcommand{\reqnomode}{\tagsleft@false}
\newcommand{\V}{\mathbb{V}}
\newcommand{\Z}{\mathcal{Z}}
\newcommand{\B}{\mathcal{B}}
\newcommand{\R}{\mathcal{R}}
\newcommand{\Ie}{\mathbb{I}_{E_\epsilon}}
\newcommand{\aeq}[1]{\begin{equation}\begin{aligned}#1\end{aligned}\end{equation}}
\newcommand{\vop}[1]{#1, \p_{#1}}
\newcommand{\M}{\mathcal{M}}
\newcommand{\h}{\mathcal{H}}
\newcommand{\Y}{\bar{Y}}

%% file: introduction.tex
When moving in a crowd, a pedestrian chooses its path based not only on its 
desired final destination but it also takes the movement of other 
surrounding pedestrians into account. The bullet points below are 
stated in \cite{naldi2010mathematical} as typical traits of pedestrian 
behavior.
\\
\begin{itemize}
 \item Will to reach specific targets. Pedestrians experience a strong 
interaction with the environment.
\item Repulsion from other individuals. Pedestrians may agree to deviate from 
their preferred path, looking for free surrounding room.
\item Deterministic if the crowd is sparse, partially random if the crowd is 
dense.\\
\end{itemize}
These properties appear in classical particle models. Other authors advocate 
smart particle models that follow decision-based dynamics. In 
\cite{naldi2010mathematical} some fundamental differences between classical and 
smart particle models are outlined. We list a few of them in Table 1.
\begin{table}[H]
\resizebox{\columnwidth}{!}{
 \begin{tabular}{ l | l}
  Classical & Smart
  \\ \hline
  Robust - interaction only through collisions & Fragile - avoidance of 
collisions and obstacles
 \\
 Blindness - dynamics ruled by inertia & Vision - dynamics ruled at least 
partially by decision
 \\
 Local - interaction is pointwise & Nonlocal - interaction at a distance
 \end{tabular}}
   \caption{}
  \label{tab:tab1}
\end{table}
\noindent
A smart particle model lets pedestrians decide blue where to walk, 
with what 
speed etc. The choice is based on some rule that takes the available 
information 
into account such as the positioning and movement of other pedestrians. 
Although more realistic, this approach has complications. If pedestrian $i$ 
moves, all pedestrians accessing information on $i$'s state might have to adapt 
their movements. The large number of connections where information is exchanged 
within a crowd is a computational difficulty. 
\\~\\
The mean-field approach to modeling crowd aversion and congestion 
for pedestrians was introduced in \cite{lachapelle2011mean}. The pedestrians 
are treated as particles following decision-based dynamics that 
optimize their path by avoiding densely crowded 
areas. Crowd aversion describes motion avoiding high density 
whereas congestion describes motion hindered by high density. The 
theory of mean-field games originates from the independent works of Lasry-Lions 
\cite{lasry2007mean} and Huang-Caines-Malham{\'e} \cite{huang2006large}. 
The cost considered in this early work is not of congestion type, i.e. the 
energy penalization is independent of the density. The framework 
was extended to several populations on the torus in \cite{feleqi2013derivation} 
and to several populations on a bounded domain with reflecting boundaries in 
\cite{cirant2015multi}, with further studies in \cite{achdou2017mean, 
bardi2017uniqueness}. Mean field games with a cost of congestion type was 
introduced by P-L. Lions in a lecture series 2011 \cite{video}. Congestion has 
also been studied in the mean-field type. In \cite{achdou2015system} the finite 
horizon case is considered. In \cite{achdou2016mean, 2016arXiv161102023A} the 
authors prove existence and uniqueness of weak solutions characterized by an 
optimization approach based on duality, and propose a numerical method for 
mean-field type control based on this result for the case of local congestion.
\\~\\
Turning to the crowd aversion model of this paper, a pedestrian 
with 
position $X^{i,N}$ in a crowd of $N$ pedestrians controls its velocity such 
that its risk 
measure, $J^{i,N}$, is minimized over a finite time horizon $[0,T]$. The risk 
measure penalizes proximity to others, energy waste and 
failure to reach a target 
area. In this paper we advocate for the use of the following nonlocal 
contribution to the risk measure, reflecting a crowd aversion 
behavior,
\begin{equation}
\label{eq:introeq}
\xN{\int_0^T \frac{1}{N-1}\sum_{\substack{j=1\\j\neq 
i}}^N\phi_r\left(X^{i,N}_t 
- X_t^{j,N}\right)dt}.
\end{equation}
\noindent
The `personal space' of a typical pedestrian is modeled by the function 
$\phi_r$ and $X^{i,N}_t - X^{j,N}_t$ is the distance between two pedestrians at 
time $t$. The personal space has support 
within a ball of radius $r$ so for positive $r$, \eqref{eq:introeq} is a 
weighted average of the crowding within the personal space and the pedestrian 
is not effected by crowding outside it. Connecting to the terminology in Table 
\ref{tab:tab1}, the case of positive $r$ will be referred to as \textit{nonlocal 
crowd aversion}. In the limit $r\rightarrow 0$ the personal space shrinks to 
a singleton and only pointwise crowding, that is collisions, will effect the 
pedestrian. This will be referred to as \textit{local crowd 
aversion}.
\\~\\
In emergency situations it is often in the interest of all pedestrians to get 
to a certain place, such an exit. In evacuation planning and 
crowd management at mass gatherings, it is in the interest of the planner to 
control the crowd along paths and towards certain areas. Common to such 
situations is the conflict between attraction to said locations 
and repulsive interactions in the crowd. Pedestrians acting under nonlocal 
crowd aversion will order themselves more densely in such places 
compared to pedestrians acting under local crowd aversion. 
This effect is caused by the larger personal space, the nonlocal 
crowd aversion term \eqref{eq:introeq} is an average over a 
bigger set hence allowing for higher densities in attractive areas. Higher 
densities will in turn allow for more effective emergency planning when 
designing for example 
escape routes. The numerical simulation in the end of this paper confirms this 
effect. The pedestrians are allowed to move freely, but the 
observed effect will become even more beneficial for a planner when introducing 
an environment for the pedestrians to interact with. In reality, crowd 
management is often done by the strategic placement of obstacles such as 
pillars and walls. Furthermore, the pedestrians acting under nonlocal 
crowd aversion travel at an overall lower risk than 
their local 
counterpart. This suggests that a crowd with nonlocal crowd averse 
behavior could potentially move at a higher velocity than its local 
counterpart which allows for faster and more 
successful evacuations.
\\~\\
In \cite{lachapelle2011mean} the mean-field optimal 
control is characterized through a matching argument. This control is an 
approximate Nash equilibrium for the crowd. It is, for each pedestrian, the 
best response to the movement of the rest of the crowd. Furthermore, two 
crowds are considered where each pedestrian has crowd-specific preferences 
such as the target location and crowd aversion preference. 
The authors set up a mean-field game and show that it is equivalent to 
an optimal control problem. In this paper, we look at the crowd from the 
bird's-eye view of an evacuation planner. We seek a `simultaneous' optimal 
strategy for all the pedestrians involved in the crowd through a mean-field 
type control approach for the single-crowd case and a mean-field 
type game approach for the multi-crowd case.
\\~\\
The contributions of this paper are the following. We identify a particle model 
 that is approximated by mean-field model for crowd aversion  
proposed in \cite{lachapelle2011mean}. This gives us insights 
into how the interaction between pedestrians in the crowd effects the 
mean-field model and reveals that the crowd of \cite{lachapelle2011mean} has a 
local crowd averse behavior. Our second contribution is a 
relaxation of the locality of the pedestrian model by allowing for interaction 
between pedestrians at a distance. Each pedestrian is given a personal space 
where it dislikes crowding, instead of interacting with other pedestrians only 
through collisions. This conceptual change is realistic since pedestrians 
do not need to be in physical contact to interact. As discussed 
above, the suggested nonlocal crowd aversion model allows for 
the following desirable features:\\
\begin{itemize}
 \item Higher densities in target areas such as exits or escape 
routes where the pedestrians have to choose between more crowding
and not reaching the target.
 \item Lower risk, which implies a potential increase in 
pedestrian velocity allowing for faster exits and a larger flow of people, a 
very useful feature in the design of evacuation strategies. \\
\end{itemize}
\noindent
Finally, we generalize the model to allow for an 
arbitrary number of interacting crowds. This multi-crowd scenario 
is treated as a mean-field type game and is linked to an 
optimal control problem, for which we prove a sufficient maximum principle. 
\\~\\
The paper is organized as follows. After a short section of 
preliminaries, we consider the single-crowd case in Section 
\ref{sec:singlecrowd}. In Section \ref{sec:multiplecrowd}, the 
multi-crowd case is studied. The results derived in Section 
\ref{sec:singlecrowd} generalize to an arbitrary finite number of interacting 
crowds and a sufficient maximum principle that characterizes the 
solution is proved. An example that highlights the difference between local and 
nonlocal crowd aversion is solved numerically in Section 
\ref{sec:numexp}. For the sake of clarity, all technical proofs are moved to an 
appendix.

%% file: preliminaries.tex
Given a general Polish space $\mathcal{S}$, let $\pp(\mathcal{S})$ denote 
the space of probability measures on $\mathcal{B}(\mathcal{S})$. For an 
element $s\in\mathcal{S}$, the Dirac measure on $s$ is an element of 
$\pp(\mathcal{S})$ and will be denoted by $\delta_s$.  Let $\pp(\mathcal{S})$ be 
equipped with the topology of weak convergence of probability measures. A metric 
that induces this topology is the bounded $1$-Lipschitz metric,
\begin{equation}
\label{eq:metric_P}
 d_{\pp(\mathcal{S})}(\mu,\nu) := \|\mu-\nu\|_1 = 
\sup_{f\in L_1}\left<\mu,f\right> - \left<\nu,f\right>,
\end{equation}
where $L_1$ is the set of real-valued functions on $\mathcal{S}$ bounded by 
$1$ and with Lipschitz coefficient 1. With his metric, $\pp(\mathcal{S})$ is a 
Polish space. The space of probability measures on 
$\mathcal{B}(\mathcal{S})$ with finite second moments will be denoted by 
$\pp_2(\mathcal{S})$,
\begin{equation}
 \pp_2(\mathcal{S}) :=\left\{\nu\in\pp(\mathcal{S})\ :\ \exists\,s_0 \in 
\mathcal{S}\text{ that satisfies } \int_{\mathcal{S}}d_{\mathcal{S}}(s, 
s_0)^2\nu(ds) < \infty\right\}.
\end{equation}
Equipped with the topology of weak convergence of measures and 
convergence of second moments, $\pp_2(\mathcal{S})$ is a Polish space. A 
compatible complete metric is the square Wasserstein metric 
$d_{\pp_2(\mathcal{S})}$, for which the following inequalities will be useful.
For all $s_i, \widetilde{s}_i  \in \mathcal{S}$ and for all $N\in\mathbb{N}$,
\begin{equation}
\label{eq:measureinequality}
 d_{\pp_2(\mathcal{S})}^2\left(\frac{1}{N}\sum_{i=1}^N\delta_{s_i}, 
\frac{1}{N}\sum_{i=1}^N\delta_{\widetilde{s}_i}\right) \leq 
\frac{1}{N}\sum_{i=1}^Nd_{\mathcal{S}}(s_i, \widetilde{s}_i)^2.
\end{equation}
For random variables $X$ and $\widetilde{X}$ with distributions $\nu$ and $\widetilde{\nu}$,
\begin{equation}
\label{eq:Wasserstein_ineq}
 d_{\pp_2(\mathcal{S})}^2(\nu,\widetilde{\nu}) \leq 
\x{|X-\widetilde{X}|^2}.
\end{equation}
Let $T>0$ be a finite time horizon and let $\mathbb{R}^d$, $d\in\N$, be 
equipped with the Euclidean norm. Let $\mathcal{M}$ and $\mathcal{M}_2$ be the 
spaces of continuous functions on 
$[0,T]$ with values in $\pp(\rd)$ and $\pp_2(\rd)$ respectively,
\begin{equation}
 \mathcal{M} := C([0,T]; \mathcal{P}(\rd)), \qquad \mathcal{M}_2 := C([0,T]; 
\mathcal{P}_2(\rd)).
\end{equation}
Equipped with the uniform metrics $d_{\mathcal{M}}$ and $d_{\mathcal{M}}$,
\begin{equation}
\label{eq:probmetric}
 d_{\mathcal{M}}(m,m') := \sup_{t\in[0,T]}d_{\pp(\rd)}(m_t, m'_t),\quad 
d_{\mathcal{M}_2}(m,m') := \sup_{t\in[0,T]}d_{\pp_2(\rd)}(m_t, m'_t),
\end{equation}
$\mathcal{M}$ and $\mathcal{M}_2$ are Polish spaces. The mathematical 
results stated above can be found in \cite[Chapter 
2]{parthasarathy1967probability} and \cite[Chapter 
14]{kallenberg2006foundations}. 
\\~\\
Let $A$ be a compact subset of $\mathbb{R}^d$. Given a filtered 
probability 
space $(\Omega, \mathcal{F}, \f, \mathbb{P})$, denote by $\A$ the set of 
$A$-valued $\f$-adapted processes such that 
\begin{equation}
\x{\int_0^T 
|a_t|^2dt} < \infty.
\end{equation}
An element of $\A$ will be called an \textit{admissible control}. From the 
context, it will be clear which stochastic basis the notation $\mathcal{A}$ is 
referring to.
\\~\\
Given a vector $x = (x^1,\dots, x^N)$ in the product space $\mathcal{S}^N$ 
and an element $y\in\mathcal{S}$, we let
\begin{equation}
 \begin{aligned}
  x^{-i} &:= (x^1,\dots, x^{i-1}, x^{i+1}, \dots, x^N),
  \\
  (y,x^{-i}) &:= (x^1,\dots, x^{i-1}, y, x^{i+1}, \dots, x^N).
 \end{aligned}
\end{equation}
Furthermore, the law of any random quantity $X$ will be 
denoted by $\mathcal{L}(X)$ and any index set of the form $\{1, \dots, N\}$ 
will be denoted by $\bN$.

%% file: singlecrowd.tex
\subsection{The particle picture}
Let $(\Omega_N, \mathcal{F}^N, \f^N, \mathbb{P}_N)$ be a 
complete filtered probability space for each $N\in\mathbb{N}$. The 
filtration $\f^N$ is right-continuous and augmented with $\mathbb{P}_N$-null 
sets. It carries the independent $d$-dimensional $\f^N$-Wiener processes 
$W^{1,N}, \dots, W^{N,N}$. Let, for each $i\in\bN$, the 
$\mathcal{F}_0^N$-measurable $\rd$-valued random variable $\xi^{i,N}$ be  
square-integrable and independent of $(W^{1,N},\dots, W^{N,N})$. Given a vector 
of admissible controls, $\bar{a}^N = (a^{1,N}, \dots, a^{N,N})\in\A^N$, 
consider the system 
\begin{equation}
 \label{eq:sdesystem1}
 dX_t^{i,N} = b(t,X^{i,N}_t,a^{i,N}_t)dt + \sigma(t,X^{i,N}_t)dW^{i,N}_t,\quad 
X_0^{i,N} = \xi^{i,N}, \quad i\in\bN.
\end{equation}

\begin{proposition}
\label{prop:exists_sol}
 Assume that
 \begin{enumerate}
\item[(A1)] $b : [0,T]\times\rd\times A \rightarrow \rd$ and 
$\sigma:[0,T]\times\rd\rightarrow\rdd$ are continuous in all 
arguments.
\item[(A2)] For all $x_1,x_2\in\rd$ and $a_1, a_2\in A$, there exists a 
constant 
$K>0$ independent of $(t,x_1,x_2,a_1,a_2)$ such that
\begin{equation*}
 \begin{aligned}
  |b(t,x_1,a_1) - b(t,x_2,a_2)| 
  &\leq 
  K(|x_1 - x_2| + |a_1 - a_2|),
  \\
 |\sigma(t,x_1) - \sigma(t,x_2)|
 &\leq K|x_1 - x_2|,
 \\
  |b(t,x_1,a_1)| + |\sigma(t,x_1)| 
  &\leq 
  K(1 + |x_1|+|a_1|).
 \end{aligned}
\end{equation*}
\end{enumerate}
 Under these assumptions, \eqref{eq:sdesystem1} has a unique 
strong 
solution in the sense that
\begin{align}
 &X_0^{i,N} = \xi^{i,N},
 \\
 \label{eq:bound_on_b}
 &\int_0^t \left|b(s, X^{i,N}_s, a^{i,N}_s)\right| + \left|\sigma(s, 
X^{i,N}_s)\right|^2ds 
< \infty,\quad t\in[0,T],\ \p-a.s.
\\
&X^{i,N}_t = \xi^{i,N} + \int_0^tb(s, X^{i,N}_s, a^{i,N}_s)ds + \int_0^t 
\sigma(s, X_s^{i,N})dW^{i,N}_s,\quad t\in[0,T].
\end{align}
Furthermore, the strong solution $X^{i,N}$ satisfies the estimate
$\xN{\sup_{s\in[0,t]}|X_s^{i,N}|^2} \leq K_t\left(1 + 
\xN{|\xi^{i,N}|^2}\right) $
for all $t\in[0,T]$, for all $i\in\bN$ and for some positive constant $K_t$ 
depending only on $t$.
\end{proposition}
\begin{proof}
A proof can be found in \cite[Chapter 1, Theorem 6.16]{yong1999stochastic}. 
Note that $K_t$ is independent of $a^{i,N}$ by compactness of $A$.
\end{proof}
\noindent
The process $X^{i,N}$ models the motion of an individual in a 
crowd of $N$ pedestrians, from now on called an $N$-crowd, who 
partially controls its velocity through the control $a^{i,N}$. Since its 
control is adapted to the full filtration $\f^N$, the model allows for the 
pedestrian to take every movement in the crowd into account. Its motion is also 
influenced by external forces, such as the random disturbance driven by 
$W^{i,N}$. The motion of the pedestrian may be modeled more generally than 
above by introducing an explicit weak interaction in the drift 
\cite{huang2006large}, 
such as
\begin{equation}
 dX_t^{i,N} = \frac{1}{N}\sum_{j=1}^N \widetilde{b}(t, X^{i,N}_t, a^{i,N}_t, 
X^{j,N})dt + \sigma(t, X^{i,N}_t)dW^{i,N}_t.
\end{equation}
It is also possible to let a common disturbance effect all pedestrians 
\cite{kolokoltsov2015mean}, to model for example evacuations during an 
earthquake, a fire, a tsunami etc.
\\~\\
Individual $i$ evaluates the state of the $N$-crowd, given by the 
control vector $\bar{a}^N = (a^{1,N},\dots, a^{N,N})$, according to its 
measure of risk
\begin{equation}
\label{eq:costfnal}
 J^{i,N}_r(\bar{a}^N) := \mathbb{E}_N\left[\int_0^T 
\left(\frac{1}{2}|a^{i,N}_t|^2 + 
\int_{\rd}\phi_r(X^{i,N}_t - y)\mu^{-i,N}_t(dy)\right) dt + 
\Psi(X_T^{i,N})\right],
\end{equation}
where $X^{1,N}, \dots, X^{N,N}$ solves \eqref{eq:sdesystem1} given 
$\bar{a}^N$ and $\mu^{-i,N}_t$ is the empirical measure of $X^{-i,N}$. The 
region where crowding has an influence on the pedestrian's 
choice of control, its 'personal space', is ideally modeled by a normalized 
indicator function,
\begin{equation}
\I_r(x) := 
\left\{
\begin{array}{l l}
\text{Vol}(B_r)^{-1}, & x \in B_r,
\\
0, & x \notin B_r, 
\end{array}
\right.
\end{equation}
where $B_r\subset \rd$ is the ball with radius $r > 0$ centered at the origin 
and Vol$(B_r)$ is its volume. The term 
\begin{equation}
 \int_{\rd}\I_r(X^{i,N}_t - y)\mu^{-i,N}_t(dy)
\end{equation}
then represents the number of pedestrians 
around $X^{i,N}_t$ within a distance less than $r$ at time 
$t$ 
\cite{djehiche2016mean}. To simplify the calculations we will use a smoothed 
version of $\I_r$. Let $\gamma_\delta$ be a mollifier,
$\gamma_\delta(x) := \gamma(x/\delta)/\delta$,
where $\gamma$ is a smooth symmetric probability density with compact support. 
For a fixed $\delta > 0$, we set
\begin{equation}
\label{eq:smooth_by_mollifier}
\phi_r(x) := \gamma_\delta \ast \I_r (x).
\end{equation}
For convergence estimates later in this section, we assume 
that the final cost $\Psi$ satisfies the following condition.\\
\begin{itemize}
\item[(A3)] For all $x_1, x_2\in\rd$ there exists a constant $K>0$ independent 
of $(x_1, x_2)$ such that
\begin{equation*}
\label{eq:Psi_lip}
 \left|\Psi(x_1) - \Psi(x_2)\right| \leq K|x_1 - x_2|.
\end{equation*}
\end{itemize}
The interpretation of the risk measure is the following. The 
first term penalizes energy usage whereas the second term 
penalizes paths through densely crowded areas. The final cost 
penalizes deviations from specific target regions. Typically the 
final cost takes large values everywhere except in areas where the pedestrians 
want to end up, places like meeting points, evacuation doors, etc.

\subsection{The mean-field type control problem}
\label{subseq:meanfieldgame}

Let $(\Omega, \F, \f, \p)$ be a complete filtered probability space such that 
the filtration is right continuous and augmented with $\p$-null sets.  
Let $\f$ carry a Wiener process $W$ and let $\xi$ be an 
$\F_0$-measurable and square-integrable $\rd$-valued random variable 
independent of $W$. Given a control $a \in \A$, the mean-field type dynamics is
\begin{equation}
\label{eq:sdeMF}
 dX_t = b(t,X_t, a_t)dt + \sigma(t,X_t)dW_t, \quad X_0 = \xi.
\end{equation}
By Proposition \ref{prop:exists_sol} there exists a unique strong solution to 
\eqref{eq:sdeMF}. The mean-field type risk measure is given by
\begin{equation}
 \label{eq:costMF}
 J_r(a) 
= 
\expv{\int_0^T \frac{1}{2}|a_t|^2 + \int_{\rd}\phi_r(X_t - y)\mu_{X_t}(dy) dt + 
\Psi(X_T)}.
\end{equation}
where $\mu_{X_t}$ is the distribution of $X_t$.
\begin{remark}
The difference between a mean-field type control and a mean-field 
game is that in general mean-field games can be reduced to a standard control 
problem and an equilibrium while a mean-field type control problem is a
nonstandard control problem 
\cite{andersson2011maximum, bensoussan2013mean}. The matching procedure 
to find the fixed point (equilibrium) for a mean-field game is pedagogically 
described as follows \cite{huang2006large,lasry2007mean}.\\
\begin{itemize}
 \item[(i)] Fix a deterministic function $\mu_t: [0,T]\rightarrow \pp_2(\rd)$.
 \item[(ii)] Solve the stochastic control problem 
 \begin{equation}
  \begin{aligned}
   &\hat{a} = \underset{a\in\A}{\text{argmin}}\ \x{\int_0^T\frac{1}{2}|a_t|^2 
+ \int_{\rd}\phi_r(X_t - 
y)\mu_t(dy)dt + \Psi(X_T)},
  \end{aligned}
 \end{equation}
 where $X$ is the dynamics corresponding to $a$.
 \item[(iii)] Determine the function $\hat{\mu}_t: [0,T]\rightarrow \pp_2(\rd)$ 
such that $\hat{\mu}_t = \mathcal{L}(\hat{X}_t)$ for all $t\in[0,T]$ where 
$\hat{X}$ is the dynamics corresponding to the optimal control $\hat{a}$.\\
\end{itemize}
In the mean-field type control setting, the measure-valued process 
$\left(\mu_{X_t}; t\in[0,T]\right)$ is not considered to be a separate 
variable but given by the input control process.
\end{remark}

\subsection{Convergence of the state process}
\label{sec:conv_of_proc}

Let the initial data $\xi^{1,N},\dots, 
\xi^{N,N}$ satisfy the following assumptions,\\
\begin{itemize}
\item[(B1)] $\sup_{N\in\mathbb{N}} 
\expvN{\frac{1}{N}\sum_{i=1}^N|\xi^{i,N}|^2} < \infty \text{ for all }i\in\bN$.
\item[(B2)] $(\xi^{1,N}, \dots, \xi^{N,N})$ is exchangeable for all $N\in\N$.
\item[(B3)] 
$\lim_{N\rightarrow\infty}\mathcal{L}\left(\frac{1}{N}\sum_{i=1}^N\delta_{\xi^{
i,N}}\right) = 
\delta_{\mu_0}$ in $\pp(\pp_2(\rd))$.\\
\end{itemize}
Under (B1)-(B3) the sequence $(\xi^{i,N})_{N\in\mathbb{N}}$ is tight
and a subsequence can be extracted that converges in distribution to a 
$\mu_0$-distributed random variable, from now on denoted by $\xi$.
We make the following assumption about the controls.\\
\begin{itemize}
 \item[(B4)] The controls are of feedback 
form, $a_t^{i,N}(\omega) = a^N(t, X^{i,N}_t(\omega))$, where each $a^N$ is an 
$A$-valued deterministic function and $a^N$ converge uniformly to $a$ as 
$N\rightarrow\infty$. Furthermore,
\begin{equation}
\label{eq:b4_equation}
 \sup_{N\in\mathbb{N}}
\expvN{\int_0^T |a^N(t,X^{i,N}_t) |^2} < \infty, \quad  \forall\ i\in\bN.
\end{equation}
\end{itemize}
\begin{remark}
Assumption (B4) implies that, while the paths of pedestrians in the 
$N$-crowd may differ, they are outcomes from a symmetric joint probability 
distribution. By exchangeability of $(\xi^{i,N}, W^{i,N})_{i=1}^N$, 
\begin{equation}
\label{eq:anonymous}
 (a^N(t, X^{i,N}_t))_{i=1}^N \overset{d}{=} (a^N(t, X^{\pi(i),N}_t))_{i=1}^N
\end{equation}
for all permutations $\pi$ of $\bN$, the interpretation is that we 
cannot distinguish between pedestrians in the crowd. The pedestrians are 
anonymous.
\end{remark}
\begin{proposition}
\label{prop:tight}
If $\mu^N$ is the empirical measure of $X^{1,N},\dots, X^{N,N}$, the solution 
of \eqref{eq:sdesystem1} given $a^N$, then $\{\mathcal{L}(\mu^N), 
N\in\mathbb{N}\}$ is tight in $\pp(\mathcal{M}_2)$.
\end{proposition}
\begin{proof}
The empirical measures are elements of $\mathcal{M}_2$ by Proposition 
\ref{prop:exists_sol} together with (B1) and (B2). The proof of tightness in 
the case of uncontrolled diffusions is found in \cite{oelschlager1985law}. The 
introduction of a control does not change the situation.
\end{proof}
\noindent
Recall that a sequence $\{X_n\}$ of random variables converges weakly to $X$ in 
a Polish space if and only if $\{X_n\}$ is tight and every convergent 
subsequence of $\{X_n\}$ converges to $X$. The tightness of 
the empirical measures implies that along a converging subsequence, $\mu^N$ 
converges weakly to the measure-valued process $\mu$ that for all 
$f\in C^2_b(\mathbb{R}^d)$ satisfies
\begin{equation}
\label{eq:inteq}
\left<\mu_t,f\right> - \left<\mu_0,f\right> = \int_0^t \left<\mu_s,b(s, 
\cdot, a(s, 
\cdot))\cdot\nabla f +\frac{1}{2}\sigma(s,\cdot)\Delta f\right>ds.
\end{equation}
Since the strong solution of \eqref{eq:sdeMF} is unique, the weak solution is 
also unique \cite{yamada1971uniqueness} which is equivalent to uniqueness 
of solutions to \eqref{eq:inteq} \cite{karatzas2012brownian}. We have the 
following result.
\begin{theorem}
 Let $X^i$, $i\in\N$, be independent copies of the strong solution of 
\eqref{eq:sdeMF}. Under assumptions (A1)-(B4), 
$X^{i,N}$ converges weakly to $X^i$ as $N\rightarrow\infty$.
\end{theorem}
\begin{proof}
Applying Sznitman's propagation of chaos theorem
\cite{sznitman1991topics}, the result follows by the weak 
convergence of 
$\mu^N$ to the deterministic measure $\mu$.
\end{proof}

\subsection{Convergence of the risk measure}
\label{sec:costfnal1}

From the previous section we know that $X^{i,N}$, the strong solution of 
\eqref{eq:sdesystem1}, converges weakly to $X$, the strong solution of 
\eqref{eq:sdeMF}, and we know that $\mu^N_t$ converges weakly to 
$\mu_{X_t}$. Applying \eqref{eq:measureinequality}, we have 
that $d_{\pp_2(\rd)}(\mu^{-i,N}_t, \mu^N_t) \leq 2/N$,
so $\mu^{-i,N}_t$ converges weakly to $\mu_{X_t}$ as well. By
Skorokhod's Representation Theorem \cite[Theorem 
3.30]{kallenberg2006foundations} we can represent (up to 
distribution) all the random variables mentioned above in a common probability 
space $(\widetilde{\Omega}, \widetilde{\F}, \widetilde{\p})$ where they 
converge $\widetilde{\p}$-almost surely. This allows us to write 
\begin{equation}
 \begin{aligned}
  &|J^{i,N}_r(a^N) - J_r(a)|
  \leq
  \mathbb{E^{\widetilde{\p}}}\Bigg[\int_0^T 
\left|\frac{1}{2}|a^{N}(t, 
X^{i,N}_t)|^2 - 
\frac{1}{2}|a(t,X_t)|^2\right|
 \\
 &\hspace{2cm}+ \Bigg|\int_{\rd}\phi_r(X^{i,N}_t - y)\mu^{-i,N}_t(dy)
 - \int_{\rd}\phi_r(X_t -y)\mu_{X_t}(dy)\Bigg|dt 
 \\
 &\hspace{2cm} + 
\left|\Psi(X^{i,N}_T) - \Psi(X_T)\right|\Bigg],
 \end{aligned}
\end{equation}
By compactness of $A$, the Continuous Mapping Theorem, (B4) and 
Dominated Convergence we have
\begin{equation}
 \lim_{N\rightarrow\infty}\mathbb{E^{\widetilde{\p}}}\Bigg[\int_0^T 
\left|\frac{1}{2}|a^{N}(t, 
X^{i,N}_t)|^2 - 
\frac{1}{2}|a(t,X_t)|^2\right|\Bigg] = 0.
\end{equation}
By (A3), Proposition \ref{prop:exists_sol} and Dominated 
Convergence,
$
\mathbb{E^{\widetilde{\p}}}[|\Psi(X^{i,N}_T) - 
\Psi(X_T)|] = 0
$
as $N \rightarrow\infty$. Note that
\begin{equation}
\begin{aligned} 
&\mathbb{E^{\widetilde{\p}}}\left[\int_0^T\left|\int_{\rd}\phi_r(X^{
i ,N}_t - 
y)\mu_t^{-i,N}(dy) - 
\int_{\rd}\phi_r(X_t - y)\mu_{X_t}(dy)\right|dt\right]
 \\
 &\qquad \leq
\mathbb{E^{\widetilde{\p}}}\left[\int_0^T\left|\int_{\rd}\phi_r(X^{i
, N}_t - 
y)\mu_t^{-i,N}(dy) - 
\int_{\rd}\phi_r(X^{i,N}_t 
- y)\mu_{X_t}(dy)\right|dt\right]
 \\
 &\qquad +
\mathbb{E^{\widetilde{\p}}}\left[\int_0^T\left|\int_{\rd}\phi_r(X^{i
, N}_t - 
y)\mu_{X_t}(dy) - 
\int_{\rd}\phi_r(X_t 
- y)\mu_{X_t}(dy)\right|dt\right].
\end{aligned}
\end{equation}
As $N\rightarrow \infty$, the first term on the right hand side tends to zero 
by the definition of weak convergence while the second tends to zero by the 
Continuous Mapping Theorem and Dominated Convergence. We have proved the 
following result.
\begin{theorem}
\label{thm:cost_convergence}
Let $a\in\A$ and $a^N = (a, \dots, a)\in\A^N$, then
$J^{i,N}_r(a^N) = J_r(a) + \ve_N$
where $\lim_{N\rightarrow \infty}\ve_N = 0$.
\end{theorem}

\subsection{Solutions to the $N$-crowd model and the MFT control 
problem}
\label{subsec:solutions-1}

The notion of solutions of the the $N$-crowd model 
(N-1) and the mean-field type control problem (MFT-1) for crowd 
aversion will now be defined.
\begin{definition}[Solution to N-1]
\label{def:solNP1}
Let $\hat{a}^N = (\hat{a}, \dots, \hat{a})\in \A^N$ for some fixed 
$\hat{a}\in\A$ and let $a^N = (a, \dots, a) \in \A^N$ for an arbitrary 
strategy $a\in\A$. Then $\hat{a}^N$ is a solution to N-1 if
\begin{equation}
 J^{i,N}_r(\hat{a}^N) \leq J^{i,N}_r(a^N), \quad \forall a\in \A,\ \forall 
i\in\bN.
\end{equation}
If, for a given $\ve>0$, $\hat{a}$ satisfies
\begin{equation}
  J^{i,N}_r(\hat{a}^N) \leq J^{i,N}_r(a^N) + \ve, \quad \forall a\in\A,\ 
\forall i\in\bN,
\end{equation}
then $\hat{a}^N$ is an $\ve$-solution to N-1.
\end{definition}
\begin{definition}[Solution to MFT-1]
 If $\hat{a}\in\A$ satisfies
 \begin{equation}
  J_r(\hat{a}) \leq J_r(a), \quad \forall a\in\A,
 \end{equation}
 then $\hat{a}$ is a solution to MFT-1.
\end{definition}
\noindent
The following result motivates the use of MFT-1 as an approximation to N-1. It 
confirms that we can construct an approximate solution to 
N-1 using a solution to MFT-1.
\begin{theorem}
\label{thm:construction_ve_sol_1}
If $\hat{a}$ solves MFT-1, then $\hat{a}^N = (\hat{a}, \dots, 
\hat{a})$ is a $\ve_N$-solution, where $\ve_N\rightarrow 0$ as 
$N\rightarrow \infty$, to N-1 among feedback strategies.
\end{theorem}
\begin{proof}
The proof follows straight away by Theorem \ref{thm:cost_convergence}.
\end{proof}
\begin{remark}
It is known that the solution 
of a mean-field game corresponds to an approximate Nash equilibrium for N-1  
(\cite{huang2006large},\cite{lasry2007mean}). To the best of our knowledge, 
this 
has not been shown to be true for solutions to mean-field type control problems. 
Theorem \ref{thm:construction_ve_sol_1} has the following interpretation; a 
mean-field type optimal control induces an approximate solution for the 
$N$-crowd if the crowd consists homogeneous pedestrians and thus a 
representative pedestrian determines the control of all. This was in fact 
visible already in Theorem \ref{thm:cost_convergence}.
\end{remark}

\subsection{Deterministic version of MFT-1}
\label{subsec:det_version}

We want to present results in a setting similar to \cite{lachapelle2011mean} to 
highlight the differences between the models. To do this, we make the 
assumption that $\mu_{X_t}$ has a density $m_X(t,\cdot)$ for all $t\in[0,T]$. 
An example of sufficient conditions for the existence is bounded drift and 
diffusion \cite{oelschlager1985law}. Under this assumption, we may 
rewrite \eqref{eq:sdeMF}-\eqref{eq:costMF} into a deterministic problem for 
$m_X$. Furthermore, an admissible control can not be stochastic in 
the deterministic problem formulation. The full stochastic problem will be 
analyzed in future work. We have a new definition of an 
admissible control.
\begin{definition}[$\Ad$]
\label{def:ad}
 A square-integrable deterministic function $a: [0,T]\times\rd\rightarrow A$ 
will be called an admissible control for the deterministic problem and the set 
of such functions is denoted by $\Ad$.
\end{definition}
\noindent
By \eqref{eq:inteq} the density $m_X$ satisfies 
\begin{equation}
 \begin{aligned}
 &\int_{\rd}f(x)m_X(t,x)dx - \int_{\rd}f(x)m_X(0,x)dx  =
 \\
 &
\int_0^t\int_{\rd}\Bigg(b(s, x, a(s, x))\cdot\nabla f(x)
 + \frac{1}{2}\tr{\sigma\sigma^T(s,x)\nabla^2 f(x)}\Bigg)m_X(s, 
x)dsdx,
 \end{aligned}
\end{equation}
for all $f\in C_b^2(\rd)$ and for all $t\in[0,T]$, hence it is a weak solution 
to
\begin{equation}
\left\{
\begin{aligned}
\label{eq:mpde1}
 \frac{\partial m_X}{\partial t}(t,x) &=  
\frac{1}{2}\tr{\nabla^2\sigma \sigma^T m_X(t,x)} - 
\nabla\cdot(b(t,x,a(t,x))m_X(t,x)),
 \\
 m_X(0,x) &= \text{ density of } \mu_0.
\end{aligned}
\right.
\end{equation}
We arrive to a deterministic version of MFT-1 (dMFT-1),
\begin{equation}
\label{eq:detNsingle}
\left\{
\begin{aligned}
&\underset{a\in\A_d}{\text{min}}& &J_r^{\text{det}}(a)
\\
&\text{s.t.}& & \frac{\partial m_X}{\partial t}(t,x) =  
\frac{1}{2}\tr{\nabla^2\sigma \sigma^T m_X(t,x)} - 
\nabla\cdot(b(t,x,a(t,x))m_X(t,x)),
 \\
 & & & m_X(0,x) = \text{ density of } \mu_0.
\end{aligned}
\right. 
\end{equation}
where 
\begin{equation}
\begin{aligned}
\label{eq:cost1}
J_r^{\text{det}}(a) &:= \int_{\rd}\int_0^T 
\Bigg(\frac{1}{2}|a(t,x)|^2m_X(t,x)\, + 
\\
&
\left(\int_{\rd}\phi_r(x-y)m_X(t,y)dy\right) m_X(t,x)\Bigg)dtdx +
\int_{\rd}\Psi(x)m(T,x)\Bigg]dx.
\end{aligned}
\end{equation}
\begin{remark}
Note that $\phi_r$ converges weakly to $\delta_0$ as $r\rightarrow 0$. 
In this limit, the risk measure tends to
\begin{equation}
 J_0^{\text{det}}(a) = \int_{\rd}\int_0^T \frac{1}{2}|a(t,x)|^2m_X(t,x) + 
m_X(t,x)^2dt + \Psi(x)m(T,x)dx,
\end{equation}
which is exactly the risk analyzed in the pedestrian crowd model of 
\cite{lachapelle2011mean}! Clearly this case corresponds to a 
situation where the pedestrian will only react to how likely it is to `bump' 
into other pedestrians. In the case of 
positive $r$, a pedestrian is effected by crowding within a personal space of 
nonzero range and reacts to the level of the density within this range. This 
is the distinction between local and nonlocal crowd averse behavior.
\end{remark}

%% file: multiplecrowd.tex
\subsection{The particle picture}
In this section, crowd averse behavior between several crowds is 
introduced. The crowds are allowed to differ in their opinions on target 
areas and/or the level of crowd aversion. This inhomogeneity is introduced 
in 
the risk measure. Let the setup be as in the previous chapter, except now 
$\f^N$ carries $NM$ independent $\f^N$-Wiener processes $W^{i,j,N}$, $i\in 
\bN$, $j\in \bM$ and 
there is for all $i\in\bN$, $j\in \bM$ a square-integrable $\F_0^N$ measurable 
$\rd$-valued random variable $\xi^{i,j,N}$ independent of all the Wiener 
processes. Given $NM$ admissible controls $a^{i,j,N}$, consider the system
\begin{equation}
\label{eq:sdeM} \left\{
\begin{aligned}
&dX^{i,j,N}_t = b(t, X^{i,j,N}_t, a^{i,j,N}_t)dt + \sigma(t, 
X^{i,j,N}_t)dW^{i,j,N}_t,
\\
&X_0^{i,j,N} = \xi^{i,j,N},\qquad i \in\bN,\ j\in\bM.
\end{aligned}
\right.
\end{equation}
In view of Proposition \ref{prop:exists_sol} there exists a unique strong 
solution to \eqref{eq:sdeM}. Pedestrian $i$ in crowd $j$ evaluates 
$\mathbf{a}$ according to its individual risk measure
\begin{equation}
\label{eq:costM}
J^{i,j,N}_{r,\Lambda}(\mathbf{a}) := \expvN{\int_0^T 
\frac{1}{2}|a^{i,j,N}|^2 + \int_{\rd}\phi_r(X^{i,j,N}_t - 
y)\widetilde{\nu}^{j,N}_{t,\Lambda}(dy) dt + 
\Psi_j(X_T^{i,j,N})},
\end{equation}
where 
\begin{equation}
\widetilde{\nu}^{j, N}_{t, \Lambda} := 
\sum_{k=1}^M\lambda_{jk}\frac{1}{N}\sum_{l=1}^N\delta_{X^{l,k, N}_t},
\end{equation}
$\lambda_{jk}$ are bounded and non-negative real numbers and $\Lambda = 
(\lambda_{jk})_{jk}$. The weights $\lambda_{jk}$ quantify the 
crowd aversion preferences in the model. If 
$\lambda_{jk}$ is high, pedestrians in crowd $j$ pay a high price for being 
close to pedestrians in crowd $k$. If $\lambda_{jk}$ is zero, pedestrians 
in crowd $j$ are indifferent to the positioning of pedestrians in crowd $k$. 
Note that if $\lambda_{jk} = 1$ for $j=k$ 
and $0$ otherwise, the crowds are disconnected in the sense that there is no 
interaction between pedestrians from different crowds.

\subsection{The mean-field type model}

Again the setup be as before except that $\f$ now carries $M$ 
independent $\f$-Wiener processes $W^j$, $j\in \bM$, and 
there are $M$ square-integrable $\F_0$ measurable $\rd$-valued random variables 
$\xi^j$, $j\in \bM$, independent of all the Wiener 
processes. Given a vector of admissible controls $\bar{a}^M = (a^1,\dots, a^M$ 
the mean-field type dynamics are
\begin{equation}
\label{eq:sdesystemMF}
 dX_t^j = b(t,X^j_t,a^j_t)dt + \sigma(t, X^j_t)dW^j_t, \quad X_0^j = \xi^j, 
\quad j\in\bM.
\end{equation}
There exists a unique strong solution to \eqref{eq:sdesystemMF} by 
Proposition \ref{prop:exists_sol}. The mean-field type 
risk measure for crowd $j\in\bM$ is given by
\begin{equation}
\label{eq:costMF-M}
 J^{j}_{r,\Lambda}(\bar{a}^M) := \mathbb{E}\left[\int_0^T\frac{1}{2}|a^j|^2 + 
\int_{\rd}\phi_r(X^j_t - y)\nu^j_{t,\Lambda}(dy)dt + \Psi_j(X_T^j)\right],
\end{equation}
where $\nu^j_{t,\Lambda} := \sum_{k=1}^M\lambda_{jk}\mu_{X^k_t}$.

\subsection{Solutions of N-M and MFT-M}

The convergence results for the single-crowd case generalizes to multiple 
crowds under the following assumptions.\\
\begin{itemize}
 \item[(C1)] $\sup_{N\in\N}\xN{\frac{1}{N}\sum_{i=1}^N|\xi^{i,j,N}|^2} < 
\infty$ for all $j\in\bM$.
 \item[(C2)] $(\xi^{1,j,N},\dots, \xi^{N,j,N})$ is exchangeable for all 
$j\in\bM$.
 \item[(C3)] $\lim_{N\rightarrow\infty}\mathcal{L}\left(\frac{1}{N}\sum_{i=1}^N 
\xi^{i,j,N}\right) = 
\delta_{\mu_0^j}$ in $\pp(\pp_2(\rd))$ for all $j\in\bM$.
\item[(C4)]  The controls are of feedback 
form, $a_t^{i,j,N}(\omega) = 
a^{j,N}(t, X^{i,j,N}_t(\omega))$ where each $a^{j,N}$ is a deterministic 
$A$-valued function and $a^{j,N}$ converge uniformly to $a^j$ as 
$N\rightarrow\infty$. Furthermore,
\begin{equation}
 \sup_{N\in\mathbb{N}} 
\expvN{\int_0^T |a^{j,N}(t,X^{i,j,N}_t) |^2} < \infty, \quad  \forall\ 
i\in\bN,\ \forall\ j\in\bM.
\end{equation}
\end{itemize}
Under (A1),(A2), (A3) for all final costs $\Psi_j$ and (C1)-(C4) 
the results from 
Section \ref{sec:conv_of_proc} and Section \ref{sec:costfnal1} immediately 
generalize to multiple crowds. Next, solutions to the $N$-crowd model (N-M) 
and the mean-field type model (MFT-M) for the multi-crowd case are 
defined.

\begin{definition}[Solution to N-M]
For any $a^j\in\A$, let $(a^{j})^N 
= (a^j, \dots, a^j)\in\A^N$. The control vector $((\hat{a}^{1})^N,\dots, 
(\hat{a}^{M})^N)$ is a solution to 
N-M if
\begin{equation}
 J^{i,j,N}_{r,\Lambda}((\hat{a}^{1})^N,\dots, (\hat{a}^{M})^N) \leq 
J^{i,j,N}_{r,\Lambda}((a^{j})^N,(\hat{a}^{-j})^N), 
\quad \forall\ a^j\in\A, \ \forall\ j\in\bM.	
\end{equation}
If
\begin{equation}
 J^{i,j,N}_{r,\Lambda}((\hat{a}^{1})^N,\dots, (\hat{a}^{M})^N) \leq 
J^{i,j,N}_{r,\Lambda}((a^{j})^N,(\hat{a}^{-j})^N) + \ve, 
\quad \forall\ a^j\in\A, \ \forall\ j\in\bM
\end{equation}
for $\ve > 0$, $((\hat{a}^{1})^N,\dots, (\hat{a}^{M})^N)$ is an 
$\ve$-solution to MFT-M.
\end{definition}
\begin{definition}[Solution to MFT-M]
 The vector $\hat{a}^M = (\hat{a}^{1,M}, \dots, 
\hat{a}^{M,M})\in\A^M$ is a solution to MFT-M if
 \begin{equation}
  J^j_{r,\Lambda}(\hat{a}^M) \leq J^j_{r,\Lambda}(a, \hat{a}^{-j,M}), \quad 
\forall\ a\in\A,\ \forall\ j\in\bM.
 \end{equation}
\end{definition}
\begin{remark}
There is a fundamental difference between the definition of solutions in the  
single-crowd case and in the multi-crowd case. The latter is a 
Nash 
equilibrium while the former is an optimal control. So, what has changed? We 
still have anonymity between pedestrians within a crowd but the vector of all 
controls used in the multi-crowd case, 
$((a^{j,N}(t,X^{i,j,N}_t))_{i=1}^N)_{j=1}^M$ for N-M and 
$(a^{j}(t,X^{j}_t))_{j=1}^M$ for MFT-M, is not exchangeable (cf. 
\eqref{eq:anonymous}). From our point of view, we may distinguish 
between two pedestrians from different crowds and hence the pedestrians are 
not anonymous anymore. Thus, it makes sense to look at a game problem 
between the crowds.
\end{remark}
\noindent
The approximation result Theorem \ref{thm:construction_ve_sol_1} 
generalizes to the multi-crowd case.
\begin{theorem}
 Assume that $\hat{a}^M$ is a solution to MFT-M. Then the vector \\
$((\hat{a}^{1,M})^N,\dots, (\hat{a}^{M,M})^N)$ is an 
$\ve_N$-solution to N-M.
\end{theorem}
\begin{proof}
 The proof follows exactly the same steps as the proof of Theorem 
\ref{thm:construction_ve_sol_1}.
\end{proof}
\noindent
Finally, under the assumption that $\mu_{X^j_t}$ admits a density 
$m_{X^j}(t,\cdot)$, we rewrite MFT-M into a deterministic problem (dMFT-M).
\begin{definition}[Solution to dMFT-M]
A control vector $\hat{a} = (\hat{a}^1,\dots,\hat{a}^M) \in\A^M_d$ solves 
dMFT-M if
\begin{equation}
 J^{j,\textup{det}}_{r,\Lambda}(\hat{a}) \leq J^{j,\textup{det}}_{r,\Lambda}(a, 
\hat{a}^{-j}),\quad \forall\ a\in\A_d, \ \forall\ j\in\bM,
\end{equation}
where
\begin{equation}
 \begin{aligned}
 J^{j,\textup{det}}_{r,\Lambda}(\hat{a}) 
 &:= 
\int_{\rd}\Bigg[\int_0^T\Bigg(\frac{1}{2}|\hat{a}^j(t,x)|^2m_{j}(t,x) 
 \\
 &+ 
\sum_{k=1}^M\lambda_{jk}\int_{\rd}\phi_r(x - y)m_k(t,y)dy\,
m_j(t,x)\Bigg)dt + \Psi_j(x)m_j(T,x)\Bigg]dx
 \end{aligned}
\end{equation}
and $m_j$ solves
\begin{equation}
 \begin{aligned}
 \left\{
 \begin{array}{l}
  \displaystyle\frac{\partial m_j}{\partial t} (t,x) = 
\frac{1}{2}\tr{\nabla^2 (\sigma\sigma^T 
m_j)(t,x)} - \nabla\cdot(b(t,x,\hat{a}^j(t,x))m_j(t,x)), 
\\
m_j(0,t) = \text{ the density of } \mu_0^j.
 \end{array}
 \right.
 \end{aligned}
\end{equation}
\end{definition}
\begin{remark}
\label{rem:r_0_M}
In the limit $r\rightarrow 0$ the risk measure is
\begin{equation}
\label{eq:cong_risk_M_r_0}
\begin{aligned}
J_{0, \Lambda}^{j,\text{det}}(a) 
&= 
\int_{\rd}\Bigg[\int_0^T\Bigg( 
\frac{1}{2}|a^j(t,x)|^2m_j(t,x) 
\\
&+ 
\sum_{k=1}^M\lambda_{j,k}m_k(t,x)m_j(t,x)\Bigg)dt + 
\Psi_j(x)m_j(T,x)\Bigg]dx. 
\end{aligned}
\end{equation}
The interpretation is the same as in the single-crowd model, when 
$r\rightarrow 0$ the personal space of the pedestrians shrink to a singleton 
and only collisions have an impact on the choice of control. 
Note that \eqref{eq:cong_risk_M_r_0} with parameters $M = 2$, $\lambda_{11} 
= \lambda_{22} =  1$ and $\lambda_{12} = \lambda_{21} = \lambda$ is exactly the 
cost that appears in \cite{lachapelle2011mean}.
\end{remark}

\subsection{An optimal control problem equivalent to dMFT-M}
In this section an optimal control problem is introduced. It is shown to have 
the same solution as dMFT-M, so instead of solving the game problem an optimal 
control is characterized by a Pontryagin-type Maximum 
Principle. To ease notation, let $\varphi = (\varphi_1,\dots,\varphi_M)$ for 
$\varphi \in \{\Psi(x), m(t,x), |a(t,x)|^2\}$. Consider the following 
optimization problem,
\begin{equation}
\tag{OC}
\label{eq:Pn}
\begin{aligned}
&\underset{a\in\A_d^M}{\text{min}}& &J_{r,\bar{\Lambda}}(a)
\\
&\text{s.t.}& & 
\displaystyle\frac{\partial m_j}{\partial t}(t,x) 
=
\frac{1}{2}\tr{\nabla^2(\sigma\sigma^T m_j)(t,x)} - 
\nabla\cdot(b(t,x,a^j(t,x))m_j(t,x)),
\\
&& &m_j(0,x) = \text{density of } \mu_0^j,\quad j\in\bM,
\end{aligned}
\end{equation}
where
\begin{align}
&J_{r,\bar{\Lambda}}(a) 
:=
\int_{\rd}\Bigg[\int_0^T \Bigg(\frac{1}{2}|a(t,x)|^2\cdot m(t,x) + 
G_{\phi_r}[m]^T(t,x)\bar{\Lambda} m(t,x) \Bigg) dt
\\
\nonumber
&\hspace{3.5cm} + \Psi(x)\cdot m(T,x)\Bigg]dx,\quad 
\bar{\Lambda}\in\mathbb{R}^{M\times M},
\\
&G_{\phi_r}[m](t,x) := \left(\int_{\rd}\phi_r(x-y)m_1(t,y)dy, \dots, 
\int_{\rd}\phi_r(x-y)m_M(t,y)dy\right),
\end{align}
The following proposition is the first link between dMFT-M and 
\eqref{eq:Pn}.

\begin{proposition}
\label{prop:prop2}
If $\hat{a}$ solves \eqref{eq:Pn} and $\Lambda = \bar{\Lambda} + 
\bar{\Lambda}^T - \textup{diag}(\bar{\Lambda})$, then $\hat{a}$ is solves 
dMFT-M.
\end{proposition}
\begin{proof}
 The proof is found in Appendix \ref{sec:proof2}.
\end{proof}
\noindent
The condition $\Lambda = \bar{\Lambda} + \bar{\Lambda}^T - 
\text{diag}(\bar{\Lambda})$ forces $\Lambda$ to be symmetric and the 
interpretation is that the aversion between crowds must be symmetric, i.e. if a 
crowd is averse to another, the other one must be equally averse towards the 
first. One can of course consider other situations, but then it is not possible 
to rewrite the game into an optimization problem on the form of \eqref{eq:Pn}. 
Therefore from now $\Lambda$ is assumed to satisfy the condition 
of Proposition \ref{prop:prop2}. Note that $\bar{\Lambda}$ does not necessarily 
have to be symmetric. Towards a characterization of the optimal control, let
\begin{equation}
\begin{aligned}
 f(t,x,a,m) 
&:= \frac{1}{2}|a(t,x)|^2\cdot m(t,x) + 
G_{\phi_r}[m](t,x)^T\bar{\Lambda} 
m(t,x),
\\
g(x,m) 
&:=
\Psi(x)\cdot m(T,x),
\end{aligned}
\end{equation}
and let, with some abuse of notation,
\begin{equation}
\begin{aligned}
&\tr{\sigma\sigma^T \nabla^2 p (t,x)} 
:= 
\left(\tr{\sigma\sigma^T \nabla^2 p_1(t,x)}, \dots, \tr{\sigma\sigma^T 
\nabla^2 p_M(t,x)}\right),
\\
&\tr{\nabla^2(\sigma\sigma^T m)(t,x)} 
:= 
\left(\tr{\nabla^2(\sigma\sigma^T m_1)(t,x)}, \dots, 
\tr{\nabla^2(\sigma\sigma^T 
m_M)(t,x)}\right). 
\end{aligned}
\end{equation}

\begin{theorem}[Sufficient maximum principle for \eqref{eq:Pn}]
\label{thm:maxp2}
Let $\hat{a} \in \A_d^M$, let
\begin{equation}
\label{eq:hamiltonian_NL}
 H(t,x,a,m,p) 
 := f(t,x,a,m) + \sum_{j=1}^M b(t,x,a^j(t,x))m_j(t,x)\cdot\nabla p_j(t,x),
\end{equation}
and let $p$ solve the adjoint equation 
\begin{equation}
\left\{
 \begin{aligned}
  \label{eq:adjoint2}
 \frac{\partial p}{\partial t}(t,x)
 &= -\Bigg(\frac{1}{2}|\hat{a}(t,x)|^2 + 
G_{\phi_r}[\hat{m}]^T(t,x)(\bar{\Lambda} + \bar{\Lambda}^T)
 \\
 & + (b(t,x,\hat{a}^1(t,x))\cdot\nabla p_1(t,x),\dots, 
b(t,x,\hat{a}^M(t,x))\cdot\nabla p_M(t,x)) 
 \\
 &+ 
 \frac{1}{2}\tr{\sigma\sigma^T\nabla^2p(t,x)}\Bigg),
\\
 p(T,x) &= \Psi(x).
 \end{aligned}
 \right.
\end{equation}
Assume that 
\begin{equation}
\label{eq:convexity_of_H}
 (a,m) \mapsto \int_{\rd} H(t,x,a,m,p)dx
\end{equation}
is convex for all $t\in[0,T]$ Then $\hat{a}$ solves \eqref{eq:Pn} if
for all $w^j\in\A_d$ and $j\in\bM$
\begin{equation}
\label{eq:optimality_condition_H}
\int_{\rd}\int_0^T D_{a^j}H(t,x,\hat{a}(t,x),\hat{m}(t,x), p)\cdot w^j(t,x)dtdx 
= 0.
\end{equation}
\end{theorem}
\noindent

\begin{proof}
Let $a,\hat{a}\in\A^M_d$ 
and let
$a_\epsilon := \epsilon a + (1-\epsilon)\hat{a}$, $\epsilon\in (0,1)$. Let  
$m^\epsilon$ and $\hat{m}$ satisfy the constraints of \eqref{eq:Pn} with 
$a_\epsilon$ and $\hat{a}$ respectively, then $\eta 
:= m^\epsilon - \hat{m}$ solves
\begin{equation}
\left\{
\begin{aligned}
 \frac{\partial\eta_j}{\partial t}(t,x) &= \frac{1}{2}\tr{\sigma^T\sigma(t,x) 
\nabla^2\eta_j(t,x)} 
  \\&- \nabla\cdot(b(t,x,\hat{a}^j(t,x))\eta_j(t,x) + 
\kappa^j_\epsilon(t,x)),
 \\
\eta_j(0,x) &= 0, \quad j\in\bM,
\end{aligned}
\right.
\end{equation}
where $\kappa^j_\epsilon := 
D_a b(t,x,\hat{a}^j(t,x))\epsilon a^j m^\epsilon_j + o(\epsilon a^j)$
is a remainder that will cancel out in the end. Let $\varphi^\epsilon(t,x,p) := 
\varphi(t, x, a_\epsilon, m^\epsilon, p)$ for $\varphi \in \{f, g, H\}$ and 
define 
$\hat{\varphi}$ in the same way using $\hat{a}$. Note that
\begin{equation}
\begin{aligned}
   f^\epsilon(t,x) - \hat{f}(t,x) 
   &= 
   H^\epsilon(t,x,p) - \hat{H}(t,x,p) 
   \\
   &- \sum_{j=1}^M 
\left(b(t,x,\hat{a}^j(t,x))\eta_j(t,x) + 
\kappa^j_\epsilon(t,x)\right)\cdot\nabla p_j(t,x)
\end{aligned}
\end{equation}
and by symmetry of $\phi_n$,
\begin{equation}
\label{eq:Gsymmetry}
 \int_{\rd} G_{\phi_r}[\hat{m}](t,x)\bar{\Lambda}\eta(t,x)dx = 
\int_{\rd} G_{\phi_r}[\eta](t,x)\bar{\Lambda}^T\hat{m}(t,x)dx.
\end{equation}
By the convexity assumption on $H$,
\begin{equation}
 \begin{aligned}
&J_{r,\bar{\Lambda}}(a_\epsilon) - J_{r,\bar{\Lambda}}(\hat{a})
\\
&=
\int_{\rd}\int_0^Tf^\epsilon(t,x) - \hat{f}(t,x)dtdx + 
\int_{\rd}g^\epsilon(x) - \hat{g}(x)dx
\\
&\geq
\int_{\rd}\int_0^T \Bigg\{ D_m\hat{H}[\eta](t,x,p) + 
\sum_{j=1}^M D_{a^j}\hat{H}(t,x,p)\cdot (a_\epsilon^j(t,x) - 
\hat{a}^j(t,x)) 
\\
&\qquad 
-\sum_{j=1}^M\left( 
b(t,x,\hat{a}^j(t,x))\eta_j(t,x)+\kappa^j_\epsilon(t,x)\right)\cdot\nabla 
p_j(t,x)\Bigg\} dtdx
\\
&\qquad 
+ \int_{\rd} \Psi(x)\cdot\eta(T,x) dx
\\
\end{aligned}
\end{equation}
By a variation argument, the $m$-derivative of $\hat{H}$ is found to be
\begin{equation}
 \begin{aligned}
&D_m\hat{H}[\eta](t,x,p) 
\\ 
&= \frac{1}{2}|\hat{a}(t,x)|^2\cdot\eta(t,x) + 
 G_{\phi_r}[\hat{m}]^T(t,x)\bar{\Lambda}\eta(t,x) 
 \\
 &\quad +\, G_{\phi_r}[\eta]^T(t,x)\bar{\Lambda}\hat{m}(t,x) + 
 \sum_{j=1}^M b(t,x,\hat{a}^j(t,x))\eta_j(t,x)\cdot\nabla p_j(t,x).
\end{aligned}
\end{equation}
The $a$-derivatives of $\hat{H}$ vanish by the optimality condition  
\eqref{eq:optimality_condition_H}. Hence, using \eqref{eq:Gsymmetry},
\begin{equation}
\begin{aligned}
&J_{r,\bar{\Lambda}}(a_\epsilon) - J_{r,\bar{\Lambda}}(\hat{a})
\\
&\geq
\int_0^T\int_{\rd} \Bigg\{\frac{1}{2}|\hat{a}(t,x)|^2 + 
G_{\phi_r}[\hat{m}]^T(t,x)(\bar{\Lambda} + \bar{\Lambda}^T) + 
\frac{1}{2}\tr{\sigma\sigma^T\nabla^2p(t,x)}
\\
&\quad + (\hat{a}^1(t,x)\cdot\nabla p_1(t,x), \dots, 
\hat{a}^M(t,x)\cdot\nabla p_M(t,x)) + \frac{\partial 
p}{\partial t}(t,x)\Bigg\}\cdot\eta(t,x)dxdt  
 \end{aligned}
\end{equation}
\noindent
Applying the adjoint equation \eqref{eq:adjoint2} now 
gives $J_{r,\bar{\Lambda}}(a_\epsilon) - 
J_{r,\bar{\Lambda}}(\hat{a}) \geq 0$ for all convex perturbations $a_\epsilon$ 
of 
$\hat{a}$. In the case of a control sets $A$ which is not convex the proof can 
be carried out in similar fashion by replacing the convex perturbation 
$a_\epsilon$ by a spike variation.
\end{proof}
\noindent
Note that if 
\begin{equation}
\label{eq:we_have_an_optimum}
\hat{a}^j(t,x) = - ( D_{a} b(t,x,a(t,x))|_{a = \hat{a}^j})\nabla p_j(t,x)
\end{equation}
the optimality condition \eqref{eq:optimality_condition_H} is satisfied. In the 
case of linear dynamics, \eqref{eq:we_have_an_optimum} is the well-known 
solution $\hat{a}^j(t,x) = -\nabla p_j(t,x)$.
No property of $\bar{\Lambda}$ except boundedness in norm was used in the proof 
of the maximum principle. The following proposition identifies all matrices 
$\bar{\Lambda}$ such that the convexity assumption 
\eqref{eq:convexity_of_H} holds.

\begin{proposition}
\label{prop:convexity2}
Condition \eqref{eq:convexity_of_H} holds if and only if
\begin{equation}
\label{eq:convexcondition2}
\int_{\rd}\int_{\rd}\phi_r(x-y)(m(t,y)-m'(t,y))^T\bar{\Lambda}(m(t,x)-m'(t,
x))dydx 
\geq 0,
\end{equation}
for all densities $m$ and $m'$ and $t\in [0,T]$ 
\end{proposition}

\begin{proof}
The convexity of $H$ in $a$ is trivial. $H$ is convex in $m$ if
\begin{equation}
\begin{aligned}
\int_{\rd}H(t,x,a,\alpha m+(1-\alpha)m',p)dtdx \leq &\ \alpha\int_{\rd} 
H(t,x,a,m,p)dtdx 
\\
&+ (1-\alpha)\int_{\rd} H(t,x,a,m',p)dtdx.
\end{aligned}
\end{equation}
The inequality above can be rearranged into
\begin{equation}
\label{eq:psdineq2}
\begin{aligned}
0 
&\geq
(\alpha^2-\alpha)\int_{\rd}G_{\phi_r}[\widetilde{m}](t,x)\bar{\Lambda}\widetilde
{m}(t,x)dx
\\
&
=
(\alpha^2-\alpha)\int_{\rd}\int_{\rd}\phi_r(x-y)\widetilde{m}(t,x)^T\bar{\Lambda
}\widetilde{m}(t,x)dydx,
\end{aligned}
\end{equation}
where $\widetilde{m} := m - m'$. The fact that $(\alpha^2 - \alpha) < 0$ 
concludes the proof.
\end{proof}
\noindent
The opposite direction of Proposition  \ref{prop:prop2} can now be proven.
\begin{proposition}
\label{prop:NgivesP2}
If $\hat{a}$ solves dMFT-M, $\hat{m}$ 
satisfies the constraints of \eqref{eq:Pn} with control $\hat{a}$ and $p$ 
satisfies the adjoint equation \eqref{eq:adjoint2}, then $\hat{a}$ solves 
\eqref{eq:Pn}.
\end{proposition}
\begin{proof}
The proof is found in Appendix \ref{app:NgivesP2}.
\end{proof}
The local risk measure, introduced in Remark \ref{rem:r_0_M}, 
will naturally yield a different Hamiltonian and adjoint equation 
than the ones above. Anyhow, results analogous to Proposition 
\ref{prop:prop2}, Theorem \ref{thm:maxp2} and Proposition \ref{prop:NgivesP2} 
hold for the local case, and their proofs are carried out following the same 
steps as in the nonlocal case. The most notable structural change is 
that in the local case, $H$ is convex if and only if 
$\bar{\Lambda}$ is positive semidefinite.

%% file: example.tex
With the following numerical example we want to illustrate the 
difference local  nonlocal crowd aversion. We consider the following simple 
pedestrian model on the one-dimensional torus $\T$,
\begin{equation}
\label{eq:numeq}
\left\{
 \begin{aligned}
  &\underset{a \in \Ad}{\text{min}}& &\int_{\T}\int_0^T 
\left\{\frac{a^2(t, 
x)}{2} + C\int_{\T}\phi_r(x - y)m(t,y)dy\right\}m(t,x) dtdx
\\
&& &\hspace{5 cm} + 
\int_{\mathbb{T}}\Psi(x)m(T,x)dx
 \\
 &\text{s.t.}& &\frac{\partial m}{\partial t}(t,x) = 
\frac{1}{2}\frac{\partial^2 m}{\partial x^2}(t,x) - \frac{\partial}{\partial 
x}(a(t,x)m(t,x)),
 \\
 & & &m(0,x) = m_0(x).
 \end{aligned}
 \right.
\end{equation}
To make the comparison we also consider the corresponding local crowd aversion 
problem
\begin{equation}
 \label{eq:numeq2}
 \left\{
  \begin{aligned}
  &\underset{a \in \Ad}{\text{minimize}}& &\int_{\T}\int_0^T 
\left\{\frac{a^2(t, 
x)}{2} + Cm(t,x)\right\}m(t,x) dt + 
\Psi(x)m(T,x)dx
 \\
 &\text{subject to}& &\frac{\partial m}{\partial t}(t,x) = 
\frac{1}{2}\frac{\partial^2 m}{\partial x^2}(t,x) - \frac{\partial}{\partial 
x}(a(t,x)m(t,x)),
 \\
 & & &m(0,x) = m_0(x).
 \end{aligned}
 \right.
\end{equation}
The constraint in \eqref{eq:numeq} and \eqref{eq:numeq2} 
corresponds to the 
dynamics of a pedestrian that controls its velocity but is disturbed by white 
noise,
\begin{equation}
 dX_t = a(t,X_t)dt + dW_t.
\end{equation}
The constant $C$ has been introduced to reweight the contribution of 
crowd aversion. By up-weighting this term, emphasis is given to 
the impact of the preference, local or nonlocal, and the difference between 
the two crowds will be more clear. To solve \eqref{eq:numeq} and 
\eqref{eq:numeq2} the gradient 
decent method (GDM) of \cite{lachapelle2011mean} is used.

\subsection{Simulations and discussions}

We let $T=1$, $C=500$ and $m_0$ and $\phi_r$ are set 
to the functions presented in Figure \ref{fig:init}. Most pedestrians are 
initially 
gathered around $x = 0$ and they have an incentive to end up around $x = 0.5$ 
at time $t=1$.
The personal space of a pedestrian is modeled as
\begin{equation}
 \hat{\phi}_{0.2}(x) := 5\I_{[0,\,.2]}(x)
\end{equation}
In the calculations, $\hat{\phi}_{0.2}$ is smoothed with a mollifier 
(cf. \eqref{eq:smooth_by_mollifier}). Note that
\begin{equation}
\int_{\T}\hat{\phi}_{0.2}(x-y)m(t,y)dy = 5\p\left(x - X_t \in [0, 0.2]\right),
\end{equation}
The use of an indicator to model the personal space thus has the 
following interpretation; the pedestrian acting under nonlocal 
crowd aversion is affected by the probability of other 
pedestrians being closer than $0.2$ from 
its own position. The averaging effect of a nonlocal crowd aversion model is 
clear: the larger the personal space, the bigger neighborhood around the 
pedestrian is affecting it.
\\~\\
The optimal controls for \eqref{eq:numeq} and 
\eqref{eq:numeq2} are found by the GDM-scheme of \cite{lachapelle2011mean}. The 
convergence of the risk is presented in Figure \ref{fig:cr_fig}.
\noindent
In Figure \ref{fig:density_fig}, a comparison between the 
solutions of \eqref{eq:numeq} and \eqref{eq:numeq2} is displayed. The crowds 
behave similarly 
until time begins to approach $t=1$. The crowd acting under nonlocal 
crowd aversion then gathers more densely in the low cost area. 
Since the crowding experienced by a pedestrian in the nonlocal 
model is an average over a larger neighborhood, it cares less about 
pointwise high densities and the benefits of reaching the low cost area 
around $x=0.5$ has a stronger impact in the nonlocal model, 
resulting in a more concentrated density. This is visualized in Figure 
\ref{fig:diff_fig}, where on the left the 
difference between crowd aversion penalties,
\begin{equation} 
\underbrace{\int_{\T}\varphi_r(x-y)m_{\text{non-local}}(t, 
y)dy}_{\text{Nonlocal crowd aversion}} - 
\underbrace{m_{\text{local}}(t,x)}_{\text{Local crowd aversion}},
\end{equation}
is plotted. On the right plot, we display
\begin{equation}
m_{\text{non-local}}(t,x) - 
m_{\text{local}}(t,x). 
\end{equation}
Note that even though the densities differ at $t = 1$, the two crowds 
experience approximately the same amount of crowding at that time 
$t=1$!
\begin{figure}[H]
\label{fig:init}
 \includegraphics[scale = 0.55, trim =  -1cm 7cm 0cm 7cm]{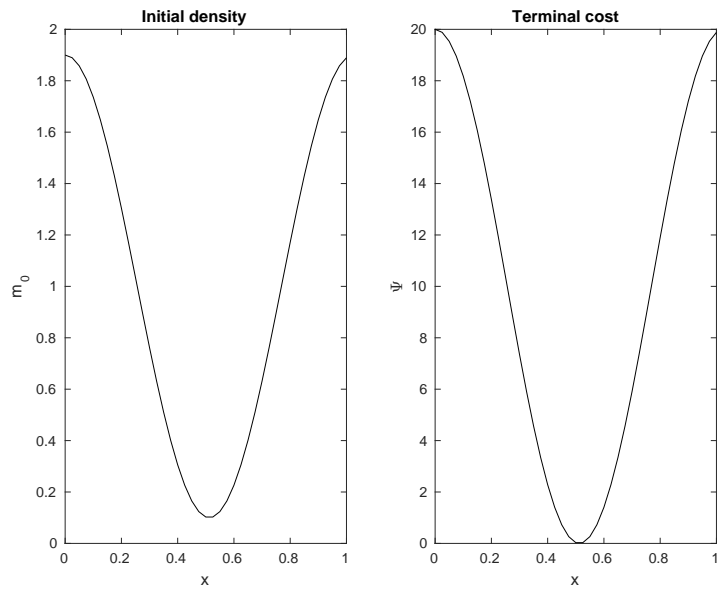}
 \caption{The initial density and terminal cost used in the simulations. 
Initially the pedestrians are crowded around $x=0$ but they will quickly 
flatten the density to heed their crowd aversion preferences. The 
low cost around $x=0.5$ will give the pedestrians an incentive to end up around 
this part of the domain at $t=1$.}
\end{figure}
\begin{figure}[H]
 \includegraphics[scale = 0.55, trim =  -1cm 7cm 0cm 7cm]{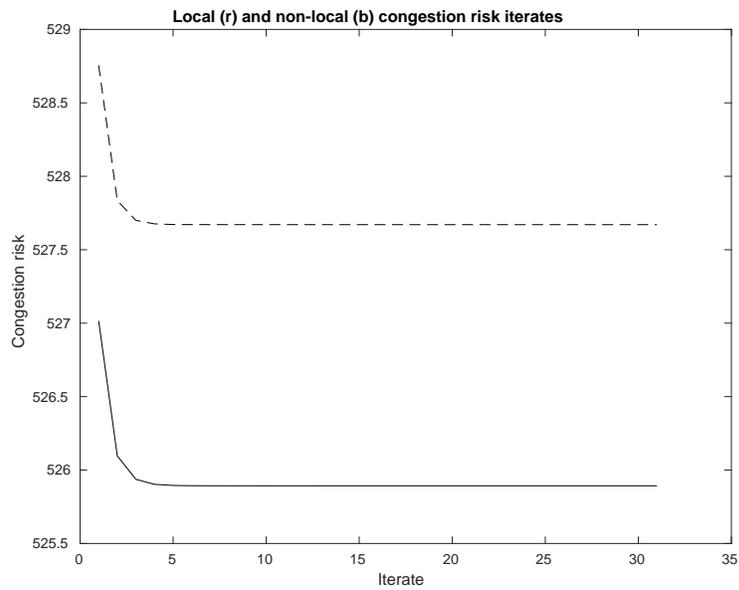}
 \label{fig:cr_fig}
 \caption{In each iteration of the GDM the control 
function $a$ is updated. The method is run until the risk 
measure, under 
local (dashed line) and nonlocal (solid line) crowd aversion, has converged to 
a minimum.}
\end{figure}
\begin{figure}[H]
\label{fig:density_fig}
 \includegraphics[scale = 0.55, trim =  -1cm 7cm 0cm 7cm]{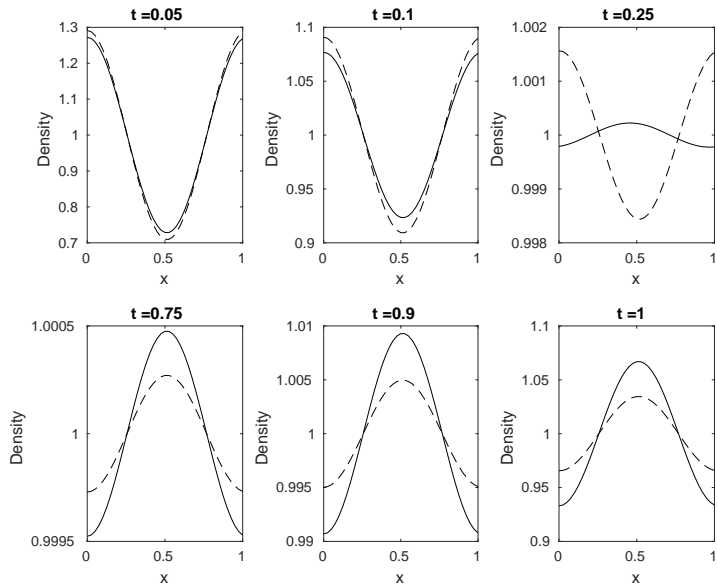}
 \caption{The optimally controlled density under local (dashed) and non-local 
(solid) crowd aversion at six instants.}
\end{figure}
\begin{figure}[H]
\label{fig:diff_fig}
 \includegraphics[scale = 0.55, trim =  -1cm 7cm 0cm 7cm]{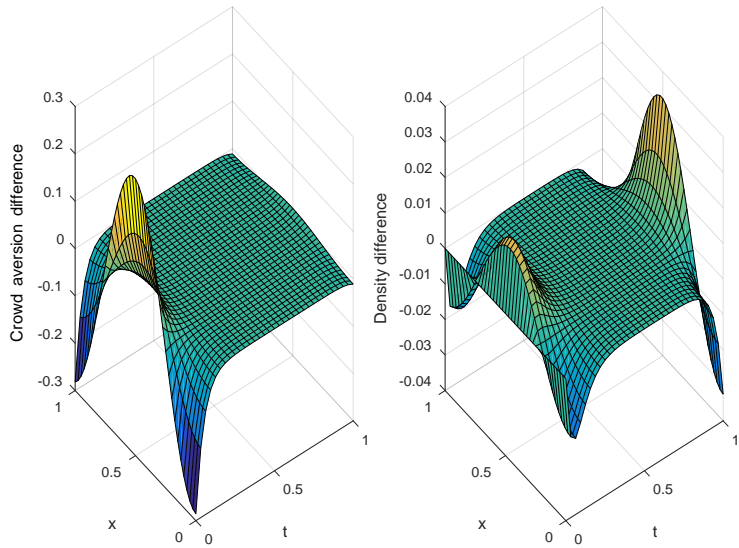}
 \caption{Differences (\textit{Non-local - Local}) between the two crowds: 
crowd aversion penalty (left plot) and density (right plot).}
\end{figure}

%% file: appendix.tex
\subsection{Proof of Proposition \ref{prop:prop2}}
The proof extends the results in \cite{lachapelle2011mean} to an arbitrary 
finite number of crowds and to nonlocal crowd aversion.
\label{sec:proof2}
\begin{proof}
Let the entries in $\bar{\Lambda}$ be denoted by 
$\bar{\lambda}_{jk}$. For each $j\in\bM$,
\begin{equation}
\begin{aligned}
&J_{r,\bar{\Lambda}}(a) - J^{j,\text{det}}_{r,\Lambda}(a^j, a^{-j})
\\
&=
\sum_{k\neq j}\left(\int_{\rd}\int_0^T \frac{1}{2}|a^k(t,x)|^2m_k(t,x)dtdx + 
\int_{\rd}\Psi_k(x)m_k(T,x)dx\right)
\\
&+
\sum_{k,l = 1}^M\left(\int_{\rd}\int_0^T 
G_{\phi_r}[m_k]^T\bar{\lambda}_{kl}m_ldtdx\right) - 
\sum_{k=1}^M\left(\int_{\rd}\int_0^T 
G_{\phi_r}[m_k]^T\lambda_{kj}m_jdtdx\right).
\end{aligned}
\label{eq:proof_neop}
\end{equation}
Note that by symmetry of $\phi$, the indices of $G_\phi[m_k]$ and $m_l$ may be 
swapped under the integral sign and the last line of \eqref{eq:proof_neop} can 
be rewritten as
\begin{equation}
\begin{aligned}
&\sum_{\substack{k,l=1\\l,k \neq j}}^M \left(\int_{\rd}\int_0^T 
G_{\phi_r}[m_k]^T\bar{\lambda}_{kl}m_l dtdx\right)
\\
&+ \int_{\rd}\int_0^T\sum_{\substack{k=1\\k\neq 
j}}^M\Big( G_{\phi_r}[m_k]^T(\bar{\lambda}_{kj} + 
\lambda_{jk} - \lambda_{kj})m_j\Big) + 
G_{\phi_r}[m_j]^T(\bar{\lambda}_{jj} - \lambda_{jj})m_jdtdx.
\end{aligned}
\end{equation}
The last line vanishes since $\Lambda = \bar{\Lambda} + \bar{\Lambda}^T - 
\text{diag}(\bar{\Lambda})$ and $J_{r,\bar{\Lambda}}(a) - 
J^{j,\text{det}}_{r,\Lambda}(a^j, a^{-j})$ is independent of $(a^j, m_j)$.
Therefore the optimality of $\hat{a}$ implies that 
\begin{equation}
 \label{eq:proof_neop2}
 J^{j,\text{det}}_{r,\Lambda}(\hat{a}) 
 \leq 
J^{j,\text{det}}_{r,\Lambda}(a^j, \hat{a}^{-j}), 
\quad
\forall\ a_j\in\A_d,\ j\in\bM.
\end{equation}
Since \eqref{eq:proof_neop2} holds for all $j\in\bM$, $\hat{a}$ is a solution 
to dMFT-M.
\end{proof}

\subsection{Proof of Proposition \ref{prop:NgivesP2}}
\label{app:NgivesP2}
This proof is a variation of \cite[Proposition 
4.2.1]{lachapelle2010quelques} which extends it to an arbitrary finite
number of crowds and to nonlocal crowd aversion.
\begin{proof}
Let, for a given $\epsilon > 0$, $a_\epsilon^j$ be the first order perturbation 
of $\hat{a}^j$ for some arbitrary $w^j$ such that
\begin{equation}
a_\epsilon^j(t,x) := \hat{a}^j(t,x) + \epsilon w^j(t,x) \in\A_d.
\end{equation}
Let $m_j^\epsilon$ satisfy the constraints in \eqref{eq:Pn} with $a^j_\epsilon$ 
and let
\begin{equation}
m_j^\epsilon(t,x) := \hat{m}_j(t,x) + \epsilon h^\epsilon_j(t,x) + 
\mathcal{O}({h^\epsilon_j}^2).
\end{equation}
Then $h^\epsilon_j$ satisfies the equation
\begin{equation}
\label{eq:pdeh_2}
\left\{
\begin{aligned}
\frac{\partial h^\epsilon_j}{\partial t}(t,x) 
&=
\frac{1}{2}\tr{\nabla^2(\sigma\sigma^T h^\epsilon_j)(t,x)} - \nabla\cdot 
\left(b(t,x,\hat{a}^j(t,x))h^\epsilon_j(t,x)\right)
\\
&- \nabla\cdot\left( \frac{b(t,x,a^j_\epsilon(t,x)) 
- b(t,x,\hat{a}^j(t,x)}{\epsilon}m^\epsilon_j(t,x))\right) ,
\\
h^\epsilon_j(0,x) 
&= 
0.
\end{aligned}
\right.
\end{equation}
Let $\mathcal{J}^j : \epsilon \rightarrow 
J^{j,\text{det}}_{r,\Lambda}(a^j_\epsilon, \hat{a}^{-j})$. Since the 
functional is convex, $\hat{a}$ solves dMFT-M if and only if
\begin{equation}
\label{eq:proofprop3_1}
\frac{\partial \mathcal{J}^j_{r,\Lambda}}{\partial \epsilon}(0) = 0,
\quad 
\forall w^j \text{ such that } \hat{a}^j + \epsilon w^j \in \A_d,\ \forall\ 
j\in\bM.
\end{equation}
Condition \eqref{eq:proofprop3_1} is equivalent to
\begin{equation}
\label{eq:proofprop3_2}
\begin{aligned}
0 = &\int_{\rd}\Bigg[\int_0^T \Bigg(\hat{a}^j(t,x)\hat{m}_j(t,x)\cdot w^j(t,x) 
+ 
\frac{1}{2}|\hat{a}^j(t,x)|^2h^0_j(t,x) 
\\
&+ 
2\lambda_{jj}\left(\int_{\rd}\phi_r(x-y)\hat{m}_j(t,y)dy\right)h^0_j(t,x) 
\\
&+ 
\sum_{k\neq j}^M \lambda_{jk}
\left(\int_{\rd}\phi_r(x-y)\hat{m}_k(t,y)dy\right)h^0_k(t,x)\Bigg)dt + 
\Psi_j(x)h^0_j(T,x)\Bigg]dx
\end{aligned}
\end{equation}
where $h^0_j$ is the solution of \eqref{eq:pdeh_2} in the limit 
$\epsilon\rightarrow 0$. Recall that $\Lambda' = \frac{1}{2}(\Lambda + 
\text{diag}(\Lambda))$. Since $p$ 
satisfies the adjoint equation, $\Psi_j(x) = p_j(T,x)$ and 
\begin{equation}
\label{eq:long_eq_with_h0}
\begin{aligned}
&\int_{\rd} p_j(T,x) h^0_j(T,x)dx
\\
&=
\int_{\rd}\int_0^T \Bigg\{-\frac{1}{2}|a^j(t,x)|^2- 
2\lambda_{jj}\left(\int_{\rd}\phi_r(x-y)\hat{m}_j(t,y)dy\right)
\\
&- \sum_{k\neq j}^M 
\lambda_{jk}\left(\int_{\rd}\phi_r(x-y)\hat{m}_k(t,y)dy\right) -
b(t,x,\hat{a}^j(t,x))\cdot\nabla p_j(t,x) 
\\
&- \frac{1}{2}\tr{\sigma\sigma^T\nabla^2p_j(t,x)}\Bigg\}h^0_j(t,x)
\\
&+ \Bigg\{\frac{1}{2}\tr{\nabla^2(\sigma\sigma^T h^0_j)(t,x)}-
\nabla\cdot\left(b(t,x,\hat{a}^j(t,x))h^0_j(t,x)\right)
\\
& - \nabla\cdot\left(
D_{a^j}b(t,x,\hat{a}^j(t,x))w^j(t,x)\hat{m}_j(t,x)\right)\Bigg\}p_j(t,x)dtdx.
\end{aligned}
\end{equation}
Inserting \eqref{eq:long_eq_with_h0} into \eqref{eq:proofprop3_2} yields
\begin{equation}
\label{eq:last_equation}
\int_{\rd}\int_0^T \left(\hat{a}^j(t,x) + 
D_{a^j}b(t,x,\hat{a}^j(t,x))^T\nabla p_j(t,x)\right)\cdot 
w_j(t,x)\hat{m_j}(t,x) dxdt = 0.
\end{equation}
Note that this since \eqref{eq:last_equation} holds for all $j\in\bM$, 
$\hat{a}$ satisfies the optimality condition in Theorem \ref{thm:maxp2} and 
therefore $\hat{a}$ is a 
solution to \eqref{eq:Pn} by 
Theorem \ref{thm:maxp2}.
\end{proof}